\documentclass[10pt,a4paper,fleqn]{article}
\usepackage{amsthm,amsmath,amssymb,amsfonts,longtable}
\usepackage{hyperref}
\newtheorem{theorem}{Theorem}

\newtheorem{example}{Example}
\newtheorem{lemma}{Lemma}

\numberwithin{equation}{section}
\numberwithin{lemma}{section}
\numberwithin{theorem}{section}
\numberwithin{corollary}{section}

\allowdisplaybreaks

\begin{document}
	\title{A note on the Lauricella matrix functions }	\author{Ravi Dwivedi\footnote {Department of Mathematics, National Institute of Technology, Kurukshetra, India. \newline E-mail: dwivedir999@gmail.com} 
		\ and Vivek Sahai\footnote {Department of Mathematics and Astronomy Lucknow University Lucknow 226007, India. \newline E-mail: sahai\_vivek@hotmail.com (Corresponding author)} }
\maketitle
\begin{abstract}
In this paper, we study the solutions of system of bilateral type matrix differential equations and presented these solutions in terms of 
Lauricella hypergeometric matrix functions of several variables and Srivastava's triple hypergeometric matrix functions of three variables. We also discuss the region of convergence and integral representation for these matrix functions. 


\end{abstract}
\section{Introduction} 
 Jod\'ar and Cort\'es \cite{jc00}, introduced the concept of a fundamental set of solutions for matrix  differential equations of the type
 \begin{align}
 X'' = f_1(z) X' + f_2(z) X f_3(z) + X' f_4(z),
 \end{align} 
 where $f_i$ are matrix valued functions of complex variable $z$. A  closed form general solution of such bilateral type matrix differential equation is determined  in terms of Gauss hypergeometric matrix function. 
 In this paper, we give the systems of partial matrix differential equations of bilateral type in the form 
  \begin{align}
  U_{x_ix_i}&= \sum_{ 
  	\substack{m,\,l=1\\ m<l}
  }^{n} \ f_{iml} (x_1, \dots, x_n)\, U_{x_m x_l} + \sum_{j=1}^{n} \  \left(U_{x_j}\, g_{ij} (x_1, \dots, x_n) + h_{ij}(x_1, \dots, x_n) \, U_{x_j}\right) \nonumber\\
& \quad + f(x_1, \dots, x_n) U g(x_1, \dots, x_n), \ 1\le i, j \le n,\label{e1.2}
  \end{align}
  where $U_{x_ix_i} = \frac{\partial^2 U}{\partial x_i^2}$, $U_{x_m x_l} = \frac{\partial^2U}{\partial x_m \partial x_l}$, $U_{x_j} = \frac{\partial U}{\partial x_j}$ and $g_{ij}$, $h_{ij}$, $f$ and $g$ are matrix valued functions of complex variables $x_1,\dots, x_n$.
  
  We show that the Lauricella matrix functions of $n$-variables satisfy the systems of bilateral type matrix differential equation of the form \eqref{e1.2}. The region of convergence and integral representation for Lauricella matrix functions of $n$-variables are also determined.   The particular case $n=3$ leading to fourteen Lauricella matrix functions of three variables has been discussed in detail.
 The sectionwise treatment is as follows:


In Section~2, we give some basic definitions from special matrix function theory that are needed in the sequel. In Section~3, we find the convergence conditions 
and system of partial matrix differential equations of bilateral type satisfied by the generalized ($n$-variable) Lauricella matrix functions. The region of convergence and integral representation satisfy by the generalized ($n$-variable) Lauricella matrix functions are also given here. In Section~4, we give the complete list of fourteen Lauricella matrix functions of three variables and matrix analogue of Srivastava triple hypergeometric  functions. Their regions of convergence, integral representations
 and the system of matrix differential equations of bilateral type satisfied by them are also given. 
 Our notations are standard.  For details, see \cite{ds1,ds4,jjc98a,jjc98b, jc00}.
 
\section{Preliminaries} 
Let the spectrum of a matrix $A$ in $\mathbb{C}^{r\times r}$, denoted by $\sigma(A)$, be the set of all eigenvalues of $A$ and let $\alpha(A) = \max \{\,\Re(z) \mid z \in \sigma(A)\,\}$ and  $\beta(A) = \min \{\,\Re(z) \mid z \in \sigma(A)\,\}$. Then for a positive stable matrix $A \in \mathbb{C}^{r \times r}$, that is $\beta(A) > 0$, the gamma matrix function is defined as \cite{jjc98a}
\[ \Gamma(A) = \int_{0}^{\infty} e^{-t} \, t^{A-I}\, dt
\]
and the reciprocal gamma matrix function is defined as \cite{jjc98a}
\begin{equation}
\Gamma^{-1}(A)= A(A+I)\dots (A+(n-1)I)\Gamma^{-1}(A+nI) , \  n\geq 1.\label{eq.07}
\end{equation}
The Pochhammer symbol  for  $A\in \mathbb{C}^{r\times r}$ is given by \cite{jjc98b}
\begin{equation}
(A)_n = \begin{cases}
I, & \text{if $n = 0$,}\\
A(A+I) \dots (A+(n-1)I), & \text{if $n\geq 1$}.
\end{cases}\label{c1eq.09}
\end{equation}
This gives
\begin{equation}
(A)_n = \Gamma^{-1}(A) \ \Gamma (A+nI), \qquad n\geq 1.\label{c1eq.010}
\end{equation} 
If $A \in \mathbb{C}^{r\times r}$ is a positive stable matrix  and $n\geq 1$ is an integer, then the  gamma matrix function can also be defined in the form of a limit as \cite{jjc98a}
\begin{equation} 
\Gamma (A) = \lim_{n \to \infty} (n-1)! \, (A)_n^{-1} \, n^A. \label{eq10} 
\end{equation}
If $A$ and $B$ are positive stable matrices in $\mathbb{C}^{r \times r}$, then the beta matrix function is defined as \cite{jjc98a}
\begin{equation}
\mathfrak{B}(A,B) =\int_{0}^{1} t^{A-I} \, (1-t)^{B-I} dt.\label{c1eq11}
\end{equation}
Furthermore, if $AB = BA$, then the beta matrix function can be written in terms of gamma matrix function as \cite{jjc98a} 
\begin{equation}
\mathfrak{B}(A,B) = \Gamma(A)\,\Gamma(B)\,\Gamma^{-1} (A+B).
\end{equation}
Using the Schur decomposition of $A$,  it follows that \cite{gl,vl}
\begin{equation}
\Vert e^{tA}\Vert \leq e^{t\alpha(A)} \sum_{k=0}^{r-1}\frac{(\Vert A\Vert r^{1/2} t)^k}{k!}, \ \ t\geq 0.\label{eq09}
\end{equation}
We shall use the notation $\Gamma \left(\begin{array}{c}
A_1, \dots, A_p \\
B_1, \dots, B_q
\end{array}\right)$ for  $\Gamma (A_1) \cdots \Gamma (A_p) \, \Gamma ^{-1}(B_1) \cdots \Gamma ^{-1} (B_q)$.

\section{Generalized Lauricella matrix functions}
The four Appell matrix functions of two variables  \cite{al,ds5} can be generalized to the following matrix functions of $n$-variables.
\begin{align}
&F_{\mathcal{A}}[A, B_1, \dots, B_n; C_1, \dots, C_n; x_1, \dots, x_n]\nonumber\\
& = \sum_{m_1, \dots, m_n = 0}^{\infty} (A)_{m_1 + \cdots + m_n} (B_1)_{m_1} \cdots (B_n)_{m_n} (C_1)_{m_1}^{-1} \cdots (C_n)_{m_n}^{-1} \, \frac{x_1^{m_1} \cdots x_n^{m_n}}{m_1! \cdots m_n!};\label{eq2.1}
\end{align}
\begin{align}
&F_{\mathcal{B}}[A_1, \dots, A_n, B_1, \dots, B_n; C; x_1, \dots, x_n]\nonumber\\
& = \sum_{m_1, \dots, m_n = 0}^{\infty} (A_1)_{m_1} \cdots (A_n)_{m_n} (B_1)_{m_1} \cdots (B_n)_{m_n} (C)^{-1}_{m_1 + \cdots + m_n}  \, \frac{x_1^{m_1} \cdots x_n^{m_n}}{m_1! \cdots m_n!};\label{eq2.2}
\end{align}
\begin{align}
&F_{\mathcal{C}}[A, B; C_1, \dots, C_n; x_1, \dots, x_n]\nonumber\\
& = \sum_{m_1, \dots, m_n = 0}^{\infty} (A)_{m_1 + \cdots + m_n} (B)_{m_1 + \cdots + m_n} (C_1)_{m_1}^{-1} \cdots (C_n)_{m_n}^{-1} \, \frac{x_1^{m_1} \cdots x_n^{m_n}}{m_1! \cdots m_n!};\label{2.3}
\end{align}
\begin{align}
&F_{\mathcal{D}}[A,  B_1, \dots, B_n; C; x_1, \dots, x_n]\nonumber\\
& = \sum_{m_1, \dots, m_n = 0}^{\infty} (A)_{m_1 + \cdots + m_n} (B_1)_{m_1} \cdots (B_n)_{m_n} (C)^{-1}_{m_1 + \cdots + m_n} \, \frac{x_1^{m_1} \cdots x_n^{m_n}}{m_1! \cdots m_n!},\label{2.4}
\end{align}
where $A$, $A_1$, $\dots$, $A_n$, $B$, $B_1$, $\dots$, $B_n$, $C$, $C_1$, $\dots$, $C_n$  are matrices in $\mathbb{C}^{r\times r}$ such that $C+kI$, $C_i + kI$, $1 \leq i \leq n$ are invertible for all integers $k \geq 0$ and $x_1, \dots, x_n$ are complex variables.

We give the system of partial matrix differential equations of bilateral type satisfied by generalized Lauricella matrix functions under certain conditions. The following theorem gives the system of partial matrix differential equations of bilateral type satisfied by the matrix function $F_{\mathcal{A}}$ defined in \eqref{eq2.1}.
\begin{theorem}\label{3.2.1}
 Let $A, B_i, C_i$ be matrices in $\mathbb{C}^{r \times r}$ such that $B_iB_j = B_jB_i$, $C_iB_j = B_jC_i$ and $C_iC_j = C_jC_i$, for each $i,j=1,\dots,n$. Then the matrix function $F_{\mathcal{A}}$ satisfies the following system of partial matrix differential equations of bilateral type
 \begin{align}
 & x_i(1-x_i) U_{x_ix_i} - x_ix_1U_{x_ix_1} - \cdots - x_ix_{i-1} U_{x_i x_{i-1}} - x_ix_{i+1} U_{x_i x_{i+1}} - \cdots - x_ix_{n} U_{x_i x_{n}} \nonumber\\
 & \quad - x_i(A+I) U_{x_i} - ({x_{1} U_{ x_{1}} + \cdots + x_{n} U_{ x_{n}}})B_i + U_{x_i} C_i - AUB_i = 0, \ i = 1, \dots, n.
 \end{align} 
\end{theorem}
\begin{proof}
	Let 
	\begin{align}
		U &= F_{\mathcal{A}}(A, B_1, \dots, B_n; C_1, \dots, C_n; x_1, \dots, x_n) = \sum_{m_1, \dots, m_n = 0}^{\infty} \mathcal{A}_{m_1, \dots, m_n} x_1^{m_1} \cdots x_n^{m_n}.
	\end{align}
	Then, we have
	\begin{align}
		&x_i(1-x_i) U_{x_ix_i} - x_ix_1U_{x_ix_1} - \cdots - x_ix_{i-1} U_{x_i x_{i-1}} - x_ix_{i+1} U_{x_i x_{i+1}} - \cdots - x_ix_{n} U_{x_i x_{n}}\nonumber\\
		& = \sum_{m_1, \dots, m_n = 0}^{\infty} (A+(m_1 + \cdots + m_n)I) \mathcal{A}_{m_1, \dots, m_n} m_i (B_i + m_i I) (C_i + m_i I)^{-1}\nonumber\\
		& \quad \times  x_1^{m_1} \cdots x_n^{m_n} - \sum_{m_1, \dots, m_n = 0}^{\infty} m_i (m_i - 1) \mathcal{A}_{m_1, \dots, m_n} x_1^{m_1} \cdots x_n^{m_n} \nonumber\\
		& \quad - \sum_{m_1, \dots, m_n = 0}^{\infty} (m_i m_1 + \cdots+ m_i m_{i-1} + m_i m_{i+1} + \cdots + m_i m_n)\mathcal{A}_{m_1, \dots, m_n} x_1^{m_1} \cdots x_n^{m_n}\nonumber\\
		& = A \sum_{m_1, \dots, m_n = 0}^{\infty} \mathcal{A}_{m_1, \dots, m_n}  x_1^{m_1} \cdots x_n^{m_n} B_i - \sum_{m_1, \dots, m_n = 0}^{\infty} (A+(m_1 + \cdots + m_n)I) \mathcal{A}_{m_1, \dots, m_n}\nonumber\\
		& \quad \times x_1^{m_1} \cdots x_n^{m_n} C_i (B_i + m_i I) (C_i + m_i I)^{-1} + (A+ I) \sum_{m_1, \dots, m_n = 0}^{\infty} \mathcal{A}_{m_1, \dots, m_n} \nonumber\\
		& \quad \times m_i x_1^{m_1} \cdots x_n^{m_n} + \sum_{m_1, \dots, m_n = 0}^{\infty} (m_i m_1 +  \cdots + m_i m_n)\mathcal{A}_{m_1, \dots, m_n} x_1^{m_1} \cdots x_n^{m_n} B_i\nonumber\\
		& =  x_i(A+I) U_{x_i} + ({x_{1} U_{ x_{1}}  + \cdots + x_{n} U_{ x_{n}}}) B_i - U_{x_i} C_i + AUB_i.
	\end{align}
	This completes the proof.
\end{proof}
In the next three theorems, we give the system of partial matrix differential equations of bilateral type satisfied by remaining three generalized Lauricella matrix functions. The proofs are similar to Theorem~\ref{3.2.1} and are omitted. 
\begin{theorem}
 Let $A_i, B_i, C$ be matrices in $\mathbb{C}^{r \times r}$ such that $A_iA_j = A_jA_i$, $B_iB_j = B_jB_i$ and $CB_j = B_jC$, for each $i,j=1,\dots,n$. Then the matrix function  $F_{\mathcal{B}}$ satisfies the following system of partial matrix differential equations of bilateral type
\begin{align}
 & x_i(1-x_i) U_{x_ix_i} + x_1U_{x_ix_1} + \cdots + x_{i-1} U_{x_i x_{i-1}} + x_{i+1} U_{x_i x_{i+1}}  + \cdots +  x_{n} U_{x_i x_{n}}\nonumber\\
& \quad - x_i(A_i+I) U_{x_i} - x_{i} U_{ x_{i}} B_i   + U_{x_i} C - A_iUB_i = 0, \ i = 1, \dots, n.
 \end{align} 
\end{theorem}
\begin{theorem}
 Let $A, B, C_i$ be matrices in $\mathbb{C}^{r \times r}$ such that $C_iC_j = C_jC_i$ and $C_jB = BC_j$, for each $i,j=1,\dots,n$. Then the matrix function $F_{\mathcal{C}}$ satisfies a system of bilateral type partial matrix differential equations
\begin{align}
&x_i(1-x_i) U_{x_ix_i} - x_1^2 U_{x_1x_1} - \cdots - x_{i-1}^2 U_{x_{i-1}x_{i-1}} - x_{i+1}^2 U_{x_{i+1}x_{i+1}} - \cdots - x_{n}^2 U_{x_{n}x_{n}} - 2x_1x_2 \nonumber\\
& \quad \times U_{x_1x_2} - \cdots - 2x_rx_sU_{x_rx_s} - \cdots - 2x_{n-1}x_nU_{x_{n-1}x_n} - (A+I) (x_1U_{x_1} + \cdots + x_nU_{x_n}) \nonumber\\
&\quad + U_{x_i} C_i - (x_1U_{x_1} + \cdots + x_nU_{x_n}) B - AUB = 0, \ (r\ne s, \ r,s,i = 1, \dots, n).
\end{align}
\end{theorem}
\begin{theorem}
 Let $A, B_i, C$ be matrices in $\mathbb{C}^{r \times r}$ such that $B_iB_j = B_jB_i$ and $CB_j = B_jC$, for each $i,j=1,\dots,n$. Then the matrix function $F_{\mathcal{D}}$ satisfies a system of bilateral type partial matrix differential equations
\begin{align}
&x_i(1-x_i) U_{x_ix_i} + x_1(1-x_i) U_{x_ix_1} + \cdots + x_{i-1}(1-x_i) U_{x_ix_{i-1}} + x_{i+1}(1-x_i) U_{x_ix_{i+1}} \nonumber\\
& \quad + \cdots + x_{n}(1-x_i) U_{x_ix_{n}} - (A+I) x_iU_{x_i}  + U_{x_i} C - (x_1U_{x_1} + \cdots + x_{n} U_{x_{n}}) B_i\nonumber\\
&\quad - AUB_i = 0, \ i = 1, \dots, n.
\end{align}
\end{theorem}

\subsection{Region of convergence}
We give here the regions of convergence of the four Lauricella matrix functions. 
We give below the proof for the Lauricella matrix function $F_{\mathcal{A}}$ and present the remaining results without proof.
\begin{theorem}\label{t3.5}
Let $A,B_1,\dots,B_n,C_1,\dots,C_n$ be positive stable matrices in $\mathbb{C}^{r \times r}$ such that $\alpha(A) < 1, \alpha(B_1) < \beta(C_1), \dots, \alpha(B_n) < \beta(C_n)$. Then the series $F_{\mathcal{A}}$ defined in \eqref{eq2.1} converges absolutely for $\vert x_1\vert  + \cdots + \vert x_n\vert < 1$.
\end{theorem}
\begin{proof}
Let $\mathcal{A}_{m_1, \dots, m_n} x_1^{m_1} \cdots x_n^{m_n}$ denote the general term of the series $F_{\mathcal{A}}$. Then 
using \eqref{eq10}, 
we get
\begin{align}
\Vert \mathcal{A}_{m_1, \dots, m_n} x_1^{m_1} \cdots x_n^{m_n}\Vert & \le N \left\Vert (m_1 + \cdots + m_n)^A\right\Vert \prod_{i=1}^{n} \left(\left\Vert  (m_i)^{B_i}\right\Vert \left\Vert  (m_i)^{-C_i}\right\Vert\right)\nonumber\\
& \quad \times \frac{(m_1 + \cdots + m_n - 1)!}{m_1! \cdots m_n!} \, \vert x_1\vert^{m_1} \cdots \vert x_n\vert^{m_n}, \label{2.7}
\end{align}
where $ N = \Vert \Gamma^{-1}(A) \Vert \Vert \Gamma^{-1}(B_1)\Vert \cdots \Vert \Gamma^{-1}(B_n)\Vert \Vert \Gamma(C_1) \Vert \cdots \Vert \Gamma(C_n) \Vert$. The Schur decomposition \eqref{eq09} 
yields
\begin{align}
&\left\Vert (m_1 + \cdots + m_n)^A\right\Vert \prod_{i=1}^{n} \left(\left\Vert  (m_i)^{B_i}\right\Vert \left\Vert  (m_i)^{-C_i}\right\Vert\right) \le S \ (m_1 + \cdots + m_n)^{\alpha(A)} \prod_{i=1}^{n} (m_i)^{\alpha(B_i) - \beta(C_i)},\label{2.8}
\end{align}
where 
\begin{align}
S & = \sum_{j= 0}^{r-1} \frac{1}{j!} \left(\Vert A\Vert r^{1/2} \ln (m_1 + \cdots + m_n)\right)^j  \prod_{i=1}^{n} \left(\sum_{j=0}^{r-1} \frac{1}{j!} \left(\max\{\Vert B_i\Vert, \Vert C_i\Vert\} r^{1/2} \ln m_i\right)^j \right)^2.
\end{align}
Using Equation \eqref{2.8} in  \eqref{2.7}, we get
\begin{align}
&\Vert \mathcal{A}_{m_1, \dots, m_n} x_1^{m_1} \cdots x_n^{m_n}\Vert\nonumber\\
& \le  N \ S \  (m_1 + \cdots + m_n)^{\alpha(A) - 1} \prod_{i=1}^{n} (m_i)^{\alpha(B_i) - \beta(C_i)} (\vert x_1\vert + \cdots + \vert x_n\vert)^{m_1 + \cdots + m_n}. 
\end{align}
Hence, for $\alpha(A) < 1, \alpha(B_1) < \beta(C_1), \dots, \alpha(B_n) < \beta(C_n)$ and $\vert x_1\vert + \cdots + \vert x_n\vert < 1$, $\Vert \mathcal{A}_{m_1, \dots, m_n}$ $ x_1^{m_1} \cdots x_n^{m_n}\Vert \rightarrow 0$ as $m_1$, $\dots$, $m_n \rightarrow \infty$. Therefore the series $F_{\mathcal{A}}$ converges absolutely.
  \end{proof}
We remark that the region of convergence of matrix functions  $F_{\mathcal{A}}$ obtained above is 
same as in \cite{ds4}, where all the matrix considered were commuting. 
Next three theorems give the regions of convergence for generalized Lauricella matrix functions  $F_{\mathcal{B}},\ F_{\mathcal{C}}$ and $F_{\mathcal{D}}$.  The proofs are similar to Theorem~\ref{t3.5} and are omitted.
\begin{theorem}
Let  $A_1$, $\dots$, $A_n$, $B_1$, $\dots$, $B_n$ and $C$ be positive stable matrices in $\mathbb{C}^{r\times r}$ such that $\alpha(A_1) + \alpha(B_1) < 2, \, \dots, \, \alpha(A_n) + \alpha(B_n) < 2,   \ \beta(C) > 1$. Then the series $F_{\mathcal{B}}$ defined in \eqref{eq2.2} converges absolutely for $\vert x_1\vert, \dots, \vert x_n\vert < 1$.	
\end{theorem}
\begin{theorem}
Let $A$, $B$, $C_1, \dots, C_n$ be positive stable matrices in $\mathbb{C}^{r\times r}$ such that $\alpha(A) + \alpha(B) < 2, \, \beta(C_1) > 1, \dots, \beta(C_n) > 1$. 
 Then the series $F_{\mathcal{C}}$ defined in \eqref{2.3} converges absolutely for $\sqrt{\vert x_1\vert} + \cdots + \sqrt{\vert x_n\vert} < 1$.	
\end{theorem}
\begin{theorem}
Let $A$, $B_1, \dots, B_n, C \in \mathbb{C}^{r\times r}$ be positive stable matrices such that $\alpha(A) < \beta(C)$,  $\alpha(B_1) < 1, \, \dots, \, \alpha(B_n) < 1$.
 Then the series $F_{\mathcal{D}}$ defined in \eqref{2.4} converges absolutely for $\vert x_1\vert$, $\dots, \vert x_n\vert < 1$. 
\end{theorem}

\subsection{Integral representations}
\begin{theorem}
  Let $A, B_i, C_i$ be matrices in $\mathbb{C}^{r \times r}$ such that $B_iB_j = B_jB_i$, $C_iB_j = B_jC_i$, $C_iC_j = C_jC_i$ and $B_i$, $C_i$, $C_i-B_i$ are positive stable for each $i,j=1,\dots,n$. Then for $\vert x_1\vert \leq r_1, \dots, \vert x_n\vert \leq r_n$, $r_1+\cdots + r_n < 1$,  the series $F_{\mathcal{A}}$ defined in \eqref{eq2.1} can be put in the integral form as
	\begin{align}
	&F_{\mathcal{A}}[A, B_1, \dots, B_n; C_1, \dots, C_n; x_1, \dots, x_n]\nonumber\\
	& = \underbrace{\int_{0}^{1}\cdots \int_{0}^{1}}_{n-fold} (1-(u_1x_1+\cdots+u_nx_n))^{-A} \ \prod_{i=1}^{n} u_i^{B_i-I} (1-u_i)^{C_i-B_i-I} \nonumber\\
	& \quad \times  du_1\cdots du_n \ \Gamma \left(\begin{array}{c}
	C_1, \dots, C_n\\
	B_1, \dots, B_n, C_1-B_1, \dots, C_n-B_n
	\end{array}\right). 
	\end{align}
\end{theorem}
\begin{theorem}\label{t2}
 Let $A_i, B_i, C$ be matrices in $\mathbb{C}^{r \times r}$ such that $B_iB_j = B_jB_i$, $CB_j = B_jC$ for each $i,j=1,\dots,n$ and $B_1, \dots, B_n$, $C$ and $C-(B_1+\dots+B_n)$ are positive stable. Then for $\vert x_1\vert, \dots, \vert x_n\vert <1$, the series $F_{\mathcal{B}}$ defined in \eqref{eq2.2}, can be put in the integral form as
	\begin{align}
	&F_{\mathcal{B}}[A_1, \dots, A_n, B_1, \dots, B_n; C; x_1, \dots, x_n]\nonumber\\
	& = \underbrace{\idotsint}_{n-fold} (1-u_1x_1)^{-A_1} \cdots (1-u_nx_n)^{-A_n} \, u_1^{B_1-I} \cdots u_n^{B_n-I} (1-(u_1+\cdots+u_n))^{C-(B_1+\cdots+B_n)-I}\nonumber\\
	&\quad \times du_1\cdots du_n \ \Gamma \left(\begin{array}{c}
	C\\
	B_1, \dots, B_n, C-(B_1+\cdots+B_n)
	\end{array}\right), \ u_1\geq 0, \dots, u_n\geq 0, \ u_1 + \cdots + u_n \leq 1.
	\end{align}
\end{theorem}
\begin{theorem}
Let $C$ be a matrix in $\mathbb{C}^{r\times r}$ such that $CB_i = B_iC$, $AC = CA$, for each $i$ and let $A$, $C$ and $C-A$ be positive stable. Then for  $\vert x_1\vert, \dots, \vert x_n\vert < 1$, the series  $F_{\mathcal{D}}$ defined in \eqref{2.4} can be put in the integral form as  
	\begin{align}
	&F_{\mathcal{D}}[A,  B_1, \dots, B_n; C; x_1, \dots, x_n]\nonumber\\
	& = \Gamma \left(\begin{array}{c}
	C\\
	A, C-A
	\end{array}\right)\int_{0}^{1} u^{A-I} (1-u)^{C-A-I} (1-ux_1)^{-B_1} \cdots (1-ux_n)^{-B_n} du.
	\end{align}
\end{theorem}
\begin{theorem}\label{t3}
 Let $A, B_i, C$ be matrices in $\mathbb{C}^{r \times r}$ such that $B_iB_j = B_jB_i$ and $CB_j = B_jC$, for each $i,j=1,\dots,n$ and let  $B_1, \dots, B_n$, $C$, $C-(B_1+\cdots+B_n)$ be positive stable. Then  for $\vert x_1\vert \leq r_1, \dots, \vert x_n\vert \leq r_n$, $r_1+\cdots + r_n < 1$, the series  $F_{\mathcal{D}}$ defined in \eqref{2.4} can be put in the integral form as
	\begin{align}
	&F_{\mathcal{D}}[A, B_1, \dots, B_n; C; x_1, \dots, x_n]\nonumber\\
	& = \underbrace{\int \cdots \int}_{n-fold} (1-(u_1x_1+\cdots+u_nx_n))^{-A} \, u_1^{B_1-I} \cdots u_n^{B_n-I} (1-(u_1+\cdots+u_n))^{C-(B_1+\cdots+B_n)-I}  \nonumber\\
	&\quad \times  du_1\cdots du_n \ \Gamma \left(\begin{array}{c}
	C\\
	B_1, \dots, B_n, C-(B_1+\cdots+B_n)
	\end{array}\right), \ u_1\geq 0, \dots, u_n\geq 0, \ u_1 + \cdots + u_n \leq 1.
	\end{align}
\end{theorem}
The proofs of these results are similar to the corresponding proofs of integral representations of  Appell matrix functions $F_2$, $F_3$ and $F_1$ respectively, \cite{ds1}. As such, the proofs are omitted. We remark that the following lemma is used in the proofs of Theorems~\ref{t2} and \ref{t3}
\begin{lemma}\cite{ds4}\label{t1.1}
Let $A_1, \dots, A_n$, $C$ be commuting matrices in $\mathbb {C}^{r \times r}$ such that $A_1, \dots, A_n$, $C, A_1 + \cdots + A_n + C$ are positive stable and $A_1 + \cdots + A_n + C + kI$ is invertible for all integers $k\geq 0$. Then
\begin{align}
\underbrace{\idotsint}_{\mbox{$n$-fold}}\, & u_1^{A_1-I} \cdots \ u_n^{A_n-I} \ (1-u_1 - \cdots - u_n)^{C-I} \, du_1 \cdots du_n = \Gamma \left(\begin{array}{c}
A_1, \dots, A_n, C\\
A_1 + \cdots + A_n + C   
\end{array}\right)\nonumber\\
&u_1\geq 0, \dots,  u_n\geq 0, u_1 + \cdots + u_n \leq 1.
\end{align}
\end{lemma}
We have not given the integral representation for generalized Lauricella matrix function $F_{\mathcal{C}}$  because the matrix function $F_{\mathcal{C}}$ does not give a simple form for integral in the arguments $x_1, \dots, x_n$.

\section{Lauricella matrix functions}
There are fourteen Lauricella matrix functions of three variables denoted by $F_1,  \dots,  F_{14}$. Out of these, $F_1$, $F_2$, $F_5$ and $F_9$ are particular cases of generalized Lauricella matrix functions $F_{\mathcal{A}}$, $F_{\mathcal{B}}$, $F_{\mathcal{C}}$ and $F_{\mathcal{D}}$ respectively for $n = 3$. 
We give below the definition of remaining ten Lauricella matrix functions, \emph{viz.}, $F_3$, $F_4$, $F_6$, $F_7$, $F_8$, $F_{10}, F_{11}, F_{12}$, $F_{13}, F_{14}$ and discuss their regions of convergence.
Let $A_i$, $B_i$ and $C_i$, $1\le i\le 3$, be matrices in $\mathbb{C}^{r \times r}$ such that each $C_i+kI$ is invertible for all integers $k \geq 0$. Then the Lauricella matrix functions are defined by
\begin{align}
&F_{3}(A_1, A_2, A_2, B_1, B_2, B_1; C_1, C_2, C_3; x, y, z) \nonumber\\
& =\sum_{m, n, p =0}^{\infty} (A_1)_{m} \, (A_2)_{n+p} \, (B_1)_{m+p} \, (B_2)_n \, (C_1)_m^{-1} \, (C_2)^{-1}_n \, (C_3)_p^{-1} \, \frac{x^m \, y^n \, z^p}{m! \, n! \, p!};\label{eq3.2}
\end{align}
\begin{align}
&F_{4}(A_1, A_1, A_1, B_1, B_2, B_2; C_1, C_2, C_3; x, y, z) \nonumber\\
& = \sum_{m, n, p =0}^{\infty} (A_1)_{m+n+p} \,  (B_1)_{m} \, (B_2)_{n+p} \, (C_1)_m^{-1} \, (C_2)^{-1}_n \, (C_3)_p^{-1} \, \frac{x^m \, y^n \, z^p}{m! \, n! \, p!};\label{eq3.3}
\end{align}
\begin{align}
&F_{6}(A_1, A_2, A_3, B_1, B_2, B_1; C_1, C_2, C_2; x, y, z) \nonumber\\
&  =\sum_{m, n, p =0}^{\infty} (A_1)_{m} \, (A_2)_{n} \, (A_3)_{p} \, (B_1)_{m+p} \, (B_2)_n \, (C_1)_m^{-1} \, (C_2)^{-1}_{n+p} \, \frac{x^m \, y^n \, z^p}{m! \, n! \, p!};\label{eq3.4}
\end{align}
\begin{align}
&F_{7}(A_1, A_2, A_2, B_1, B_2, B_3; C_1, C_1, C_1; x, y, z) \nonumber\\
& =\sum_{m, n, p =0}^{\infty} (A_1)_{m} \, (A_2)_{n+p} \, (B_1)_{m} \, (B_2)_n \, (B_3)_{p} \, (C_1)_{m+n+p}^{-1} \, \frac{x^m \, y^n \, z^p}{m! \, n! \, p!};\label{eq3.5}
\end{align}
\begin{align}
&F_{8}(A_1, A_1, A_1, B_1, B_2, B_3; C_1, C_2, C_2; x, y, z) \nonumber\\
& = \sum_{m, n, p =0}^{\infty} (A_1)_{m+n+p} \, (B_1)_{m} \, (B_2)_n \, (B_3)_p \, (C_1)_m^{-1} \, (C_2)^{-1}_{n+p} \, \frac{x^m \, y^n \, z^p}{m! \, n! \, p!};\label{eq3.6}
\end{align}
\begin{align}
&F_{10}(A_1, A_2, A_1, B_1, B_2, B_1; C_1, C_2, C_2; x, y, z) \nonumber\\
& = \sum_{m, n, p =0}^{\infty} (A_1)_{m+p} \, (A_2)_n \, (B_1)_{m+p} \, (B_2)_n  (C_1)_m^{-1} \, (C_2)^{-1}_{n+p} \, \frac{x^m \, y^n \, z^p}{m! \, n! \, p!};\label{eq3.7}
\end{align}
\begin{align}
&F_{11}(A_1, A_2, A_2, B_1, B_2, B_1; C_1, C_2, C_2; x, y, z) \nonumber\\
& = \sum_{m, n, p =0}^{\infty} (A_1)_{m} \, (A_2)_{n+p} \, (B_1)_{m+p} \, (B_2)_n  (C_1)_m^{-1} \, (C_2)^{-1}_{n+p} \, \frac{x^m \, y^n \, z^p}{m! \, n! \, p!};\label{eq3.8}
\end{align}
\begin{align}
&F_{12}(A_1, A_2, A_1, B_1, B_1, B_2; C_1, C_2, C_2; x, y, z) \nonumber\\
& = \sum_{m, n, p =0}^{\infty} (A_1)_{m+p} \, (A_2)_{n} \, (B_1)_{m+n} \, (B_2)_p \, (C_1)_m^{-1} \, (C_2)^{-1}_{n+p} \, \frac{x^m \, y^n \, z^p}{m! \, n! \, p!};\label{eq3.9}
\end{align}
\begin{align}
&F_{13}(A_1, A_2, A_2, B_1, B_2, B_1; C_1, C_1, C_1; x, y, z) \nonumber\\
& = \sum_{m, n, p =0}^{\infty} (A_1)_{m} \, (A_2)_{n+p} \, (B_1)_{m+p} \, (B_2)_n \, (C_1)^{-1}_{m+n+p} \, \frac{x^m \, y^n \, z^p}{m! \, n! \, p!};\label{eq3.10}
\end{align}
\begin{align}
&F_{14}(A_1, A_1, A_1, B_1, B_2, B_1; C_1, C_2, C_2; x, y, z) \nonumber\\
& = \sum_{m, n, p =0}^{\infty} (A_1)_{m+n+p} \, (B_1)_{m+p} \, (B_2)_n \, (C_1)_m^{-1} \, (C_2)^{-1}_{n+p} \, \frac{x^m \, y^n \, z^p}{m! \, n! \, p!}.\label{eq3.11}
\end{align}
The matrix analogues of the three Srivastava's triple hypergeometric functions,  \cite{hm84}, are given by
\begin{align}
&H_{\mathcal{A}}(A, B, B'; C, C'; x, y, z)  = \sum_{m, n, p =0}^{\infty} (A)_{m+p} \, (B)_{m+n} \, (B')_{n+p} \, (C)_m^{-1} \, (C')^{-1}_{n+p} \, \frac{x^m \, y^n \, z^p}{m! \, n! \, p!},\label{eq3.12}
\\[5pt]
&H_{\mathcal{B}}(A, B, B'; C, C', C''; x, y, z) \nonumber\\
& = \sum_{m, n, p =0}^{\infty} (A)_{m+p} \, (B)_{m+n} \, (B')_{n+p} \, (C)_m^{-1} \, (C')^{-1}_{n} \, (C'')_{p}^{-1} \, \frac{x^m \, y^n \, z^p}{m! \, n! \, p!},\label{eq3.13}
\\[5pt]
&H_{\mathcal{C}}(A, B, B'; C; x, y, z)  = \sum_{m, n, p =0}^{\infty} (A)_{m+p} \, (B)_{m+n} \, (B')_{n+p} \,  (C)^{-1}_{m+n+p} \, \frac{x^m \, y^n \, z^p}{m! \, n! \, p!},\label{eq3.14}
\end{align}
where $A$, $B$, $B'$, $C$, $C'$ and $C''$ are matrices in $\mathbb{C}^{r\times r}$ such that $C+kI$, $C'+kI$ and $C''+kI$ are invertible for all integers $k\geq 0$. 

We remark that  the regions of convergence of Lauricella matrix functions and Srivastava's triple hypergeometric matrix functions defined above remain the same as in \cite{ds4}, with proofs modified as illustrated in Theorem~\ref{t3.5}. As such, these results are omitted.
The following theorem gives the region of convergence of Lauricella matrix function $F_3$.
\begin{theorem}
Let $A_1$, $A_2$, $B_1$, $B_2$, $C_1$, $C_2$ and $C_3$ be positive stable matrices in $\mathbb{C}^{r \times r}$ such that $\alpha(A_1) < \beta(C_1), \ \alpha(A_2) < 1, \ \alpha(B_1) < 1, \ \alpha(B_2) < \beta(C_2), \  \beta(C_3) > 1$. Then the series $F_3$ defined in \eqref{eq3.2} converges absolutely for $\vert x\vert < r, \ \vert y \vert < s, \ \vert z\vert < t, \ (1-r)(1-s) = t$.
\end{theorem}
\begin{proof}
Let $\mathcal{A}_{m,n,p} x^{m} y^n z^p$ denote the general term of the series $F_{3}$. Then 
\begin{align}
&\Vert \mathcal{A}_{m,n,p} x^{m} y^n z^p\Vert \nonumber\\
& \le \left\Vert (A_1)_{m} \right\Vert  \left\Vert (A_2)_{n+p} \right\Vert  \Vert (B_1)_{m+p} \Vert  \left\Vert (B_2)_{n} \right\Vert \Vert (C_1)_{m}^{-1} \Vert \Vert(C_2)_{n}^{-1}\Vert  \Vert(C_3)_{p}^{-1}\Vert \left \vert  \frac{x^{m} y^n z^p}{m! n! p!}\right\vert\nonumber\\
& \le \left\Vert \frac{(A_1)_{m} \,m^{A_1}\, m^{-A_1} \,(m-1)!}{(m-1)!} \right\Vert  \left\Vert\frac{(A_2)_{n+p} \,(n+p)^{A_2} \,(n+p)^{-A_2} \,(n+p-1)!}{(n+p-1)!} \right\Vert\nonumber\\
&\quad \times   \left\Vert\frac{(B_1)_{m+p} \,(m+p)^{B_1} \,(m+p)^{-B_1} \,(m+p-1)!}{(m+p-1)!} \right\Vert  \left\Vert \frac{(B_2)_{n} \,n^{B_2}\, n^{-B_2}\, (n-1)!}{(n-1)!}  \right\Vert\nonumber\\
&\quad \times \left\Vert \frac{(C_1)^{-1}_{m} \,m^{C_1}\, m^{-C_1}\, (m-1)!}{(m-1)!}  \right\Vert \left\Vert \frac{(C_2)^{-1}_{n} \,n^{C_2}\, n^{-C_2}\, (n-1)!}{(n-1)!}  \right\Vert \left\Vert \frac{(C_3)_{p} \ p^{C_3}\, p^{-C_3}\, (p-1)!}{(p-1)!}  \right\Vert \nonumber\\
& \quad \times \frac{\vert x\vert^{m} \vert y\vert^n \vert z\vert^p}{m! n! p!}.\label{4.16} 
\end{align}
Using $\Gamma (A) = \lim_{n \to \infty} (n-1)! \, (A)_n^{-1} \, n^A$, \cite{jjc98a}, we get
\begin{align}
\Vert \mathcal{A}_{m,n,p} x^{m} y^n z^p\Vert & \le N \Vert m^{A_1}\Vert \Vert (n+p)^{A_2}\Vert \Vert (m+p)^{B_1} \Vert \Vert n^{B_2}\Vert \Vert m^{-C_1}\Vert \Vert n^{-C_2}\Vert \Vert p^{-C_3}\Vert\nonumber\\
& \quad \times  \frac{(n+p-1)! \, (m+p-1)!}{m! \,  n! \, p! \,  (p-1)!} \, \vert x\vert^m \, \vert y\vert^n \, \vert z\vert^p,\label{4.17}
\end{align}
where $ N = \Vert \Gamma^{-1}(A_1) \Vert \Vert \Gamma^{-1}(A_2) \Vert\Vert \Gamma^{-1}(B_1)\Vert \Vert \Gamma^{-1}(B_2)\Vert \Vert \Gamma(C_1) \Vert \Vert \Gamma(C_2) \Vert \Vert \Gamma(C_3) \Vert$. The Schur decomposition, \cite{gl,vl}, yields
\begin{align}
&\Vert m^{A_1}\Vert \Vert (n+p)^{A_2}\Vert \Vert (m+p)^{B_1} \Vert \Vert n^{B_2}\Vert \Vert m^{-C_1}\Vert \Vert n^{-C_2}\Vert \Vert p^{-C_3}\Vert\nonumber\\
& \le S \,m^{\alpha(A_1) - \beta(C_1)} \,n^{\alpha(B_2) - \beta(C_2)} \,p^{-\beta(C_3)}\, (m+p)^{\alpha(B_1)} \,(n+p)^{\alpha(A_2)},\label{4.18}
\end{align}
where 
\begin{align}
S & = \left(\sum_{j= 0}^{r-1} \frac{1}{j!} \left(\max\{\Vert A_1\Vert, \Vert C_1\Vert\} r^{1/2} \ln m\right)^j\right)^2 \left(\sum_{j=0}^{r-1} \frac{1}{j!} \left(\max\{\Vert B_2\Vert, \Vert C_2\Vert\} r^{1/2} \ln n\right)^j \right)^2 \nonumber\\
& \quad \times  \sum_{j= 0}^{r-1} \frac{1}{j!} \left(\Vert C_3\Vert r^{1/2} \ln p\right)^j \sum_{j= 0}^{r-1} \frac{1}{j!} \left(\Vert A_2\Vert r^{1/2} \ln (n+p)\right)^j \sum_{j= 0}^{r-1} \frac{1}{j!} \left(\Vert B_2\Vert r^{1/2} \ln (m+p)\right)^j.
\end{align}
Equations \eqref{4.17} and \eqref{4.18} gives
\begin{align}
\Vert  \mathcal{A}_{m,n,p} x^{m} y^n z^p\Vert & \le N \ S \  m^{\alpha(A_1) - \beta(C_1)} \,n^{\alpha(B_2) - \beta(C_2)} \,p^{1-\beta(C_3)}\, (m+p)^{\alpha(B_1) -1} \nonumber\\
& \quad \times (n+p)^{\alpha(A_2) - 1} \, \frac{(n+p)! \, (m+p)!}{m! \,  n! \, p! \,  p!} \, \vert x\vert^m \, \vert y\vert^n \, \vert z\vert^p. 
\end{align}
Hence, for  $\alpha(A_1) < \beta(C_1), \ \alpha(A_2) < 1, \ \alpha(B_1) < 1, \ \alpha(B_2) < \beta(C_2), \  \beta(C_3) > 1$ and $\vert x\vert < r, \ \vert y \vert < s, \ \vert z\vert < t, \ (1-r)(1-s) = t$, $\Vert \mathcal{A}_{m,n,p} x^m y^n z^p\Vert \rightarrow 0$ as $ m, n, p \rightarrow \infty$. Therefore the series $F_3$ converges absolutely.
\end{proof}
We now find the system of partial matrix differential equations of bilateral type obeyed by the Lauricella matrix function $F_3$.
\begin{theorem}\label{th5.2}
Let $A_1$, $A_2$, $B_1$, $B_2$, $C_1$, $C_2$, $C_3$ be matrices in $\mathbb{C}^{r \times r}$ such that $A_1A_2 = A_2A_1$, $B_1B_2 = B_2B_1$, $B_iC_j = C_jB_i$ and $C_iC_j = C_jC_i$, for each $i,j=1,2,3$. Then the system of partial matrix differential equations of bilateral type satisfied by the matrix function $F_3$ defined in \eqref{eq3.2} is
\begin{align}
&x(1-x)U_{xx} - xzU_{xz} + U_x(C_1 - (B_1+I)x) - A_1(xU_x + zU_z) - A_1UB_1 = 0,\label{4.21}
\\[5pt]
&y(1-y)U_{yy} - yzU_{yz} + U_y(C_2 - (B_2+I)y) - zU_zB_2 - yA_2U_y - A_2UB_2 = 0,\label{4.22}
\\[5pt]
&z(1-z)U_{zz} - (xyU_{xy} + yzU_{yz} + xzU_{xz}) + U_z(C_3-(B_1+I)z) -yU_yB_1\nonumber
\\[5pt]
&- A_2(xU_x + zU_z) - A_2UB_1 = 0.\label{4.23}
\end{align}
\end{theorem}
\begin{proof}
Let $U = F_{3}(A_1, A_2, A_2, B_1, B_2, B_1; C_1, C_2, C_3; x, y, z) = \sum_{m,n, p=0}^{\infty} U_{m,n, p} \, x^m y^n z^p$. Since $U(x, y, z)$ converges absolutely for $\alpha(A_1) < \beta(C_1), \ \alpha(A_2) < 1, \ \alpha(B_1) < 1, \ \alpha(B_2) < \beta(C_2), \  \beta(C_3) > 1$ and $\vert x\vert < r, \ \vert y \vert < s, \ \vert z\vert < t, \ (1-r)(1-s) = t$, so it is termwise differentiable in this domain and 
\begin{align}
&U_x = \sum_{m=1,n, p=0}^{\infty} m\, U_{m,n,p} \, x^{m-1} y^n z^p,\  U_y = \sum_{m=0,n=1, p=0}^{\infty} n\, U_{m,n,p} \, x^m y^{n-1} z^p, \nonumber\\
& U_z = \sum_{m=0,n=0, p=1}^{\infty} p\, U_{m,n,p} \, x^m y^{n} z^{p-1}, \ U_{xx} = \sum_{m=2,n,p=0}^{\infty}m\, (m-1)\, U_{m,n,p} \, x^{m-2} y^n z^p,\nonumber\\
& U_{xy} = \sum_{m=1,n = 1,p=0}^{\infty} m\,n\, U_{m,n,p} \, x^{m-1} y^{n-1} z^{p}, \ U_{xz} = \sum_{m=1,n = 0,p=1}^{\infty} m\,p\, U_{m,n,p} \, x^{m-1} y^{n} z^{p-1},\nonumber\\
& U_{yy} = \sum_{m=0,n=2,p=0}^{\infty}n\, (n-1)\, U_{m,n,p} \, x^{m} y^{n-2} z^p, \ U_{yz} = \sum_{m=0,n=1,p=1}^{\infty}n\, p\, U_{m,n,p} \, x^{m} y^{n-1} z^{p-1},\nonumber\\
& U_{zz} = \sum_{m=0,n=0,p=2}^{\infty}p\, (p-1)\, U_{m,n,p} \, x^{m} y^{n} z^{p-2}.\label{e3.17}
\end{align}  
This gives  
\begin{align}
&x(1-x)U_{xx} - xzU_{xz}\nonumber\\
& = \sum_{m,n, p=0}^{\infty} m (m+1) U_{m+1, n, p}\, x^m y^n z^p - \sum_{m,n,p=0}^{\infty} m(m-1) U_{m,n,p}\, x^m y^n z^p\nonumber\\
& \quad - \sum_{m,n,p=0}^{\infty} m p \, U_{m,n,p}\, x^m y^n z^p\nonumber\\
& = \sum_{m,n, p=0}^{\infty} (A_1+mI) U_{m,n,p} (B_1+(m+p)I) x^m y^n z^p - \sum_{m,n,p=0}^{\infty} m(m-1) U_{m,n,p}\, x^m y^n z^p \nonumber\\
& \quad - \sum_{m,n,p=0}^{\infty} (A_1+mI) U_{m,n,p} (B_1+(m+p)I) (C_1+mI)^{-1} \, C_1 x^m y^n z^p\nonumber\\
&\quad  - \sum_{m,n,p=0}^{\infty} m p \,U_{m,n,p}\, x^m y^n z^p\nonumber\\
& = \sum_{m,n=0}^{\infty} (m+p) A_1 U_{m,n,p} \, x^m y^n z^p + \sum_{m,n,p=0}^{\infty} m \, U_{m,n,p}\, x^m y^n z^p (B_1+I) \nonumber\\
& \quad -  \sum_{m,n,p=0}^{\infty} (A_1+mI) U_{m,n,p} (B_1 +(m+p)I) (C_1+mI)^{-1} \, C_1 x^m y^n z^p \nonumber\\
& \quad  +  \sum_{m,n,p=0}^{\infty} A_1 U_{m,n,p} x^m y^n z^p B_1\nonumber\\
& =  - U_x(C_1 - (B_1+I)x) + A_1(xU_x + zU_z) + A_1UB_1,
\end{align}
completing the proof of Equation \eqref{4.21}. Similarly we are able to show that the  matrix function $F_3$ satisfies the bilateral type matrix differential equations \eqref{4.22} and \eqref{4.23}. 
\end{proof}
The 
bilateral type partial matrix differential equations obeyed by remaining Lauricella matrix functions and Srivastava's triple hypergeometric matrix functions are tabulated in Table~1. 

{\renewcommand\arraystretch{1.5}
	\begin{longtable}{|l|l|c|}
		\hline
		Functions & Systems of Matrix Differential equations & Conditions on Matrices\\
		\hline
$F_4$ & $\begin{array}{c}
x(1-x) U_{xx} -(xyU_{xy} + xzU_{xz}) - (yU_y + zU_z) B_1\\
 + U_x (C_1-(B_1+I)x) - xA_1 U_x - A_1 U B_1 = 0,\\
y(1-y) U_{yy} -(xyU_{xy} + xzU_{xz} + 2yz U_{yz} + z^2 U_{zz}) \\
 - xU_x B_2 + U_y (C_2-(B_2+I)y) - A_1 (yU_y + zU_z) \\
 -zU_z (B_2 + I)- A_1 U B_2 = 0,\\
z(1-z) U_{zz} -(xyU_{xy} + xzU_{xz} + 2yz U_{yz} + y^2 U_{yy}) \\
- xU_x B_2 + U_z (C_3-(B_2+I)z) - A_1 (yU_y + zU_z) \\
-yU_y (B_2 + I)- A_1 U B_2 = 0;
		\end{array}$ &  $\begin{array}{c}
	B_1B_2 = B_2B_1,\\
	 C_iC_j = C_jC_i,\\
	B_i C_j = C_j B_i 
		\end{array}$\\
		\hline
		$F_6$ & $\begin{array}{c}
		x(1-x) U_{xx} - xz U_{xz} + U_x (C_1 - (B_1 + I)x) \\
		- A_1(xU_x + zU_z) - A_1 U B_1 = 0,\\
y(1-y) U_{yy} + z U_{yz} + U_y (C_2 - (B_2 + I)y) \\
- yA_2 U_y - A_2 U B_2 = 0,\\
z(1-z) U_{zz} - xz U_{xz} + y U_{yz} + U_z (C_2 - (B_1 + I)z) \\
- A_3(xU_x + zU_z) - A_3 U B_1 = 0;  
		\end{array}$ &  $\begin{array}{c}
		A_i A_j = A_j A_i\\
		B_1B_2 = B_2B_1\\
		  C_iC_j = C_jC_i,\\
		 B_i C_j = C_j B_i 
		\end{array}$\\
		\hline
		$F_7$ & $\begin{array}{c}
		x(1-x)U_{xx} + y U_{xy} + z U_{xz} + U_x (C_1 - (B_1 + I)x)\\
		-  xA_1 U_x - A_1 U B_1 = 0,\\
			y(1-y)U_{yy} + x U_{xy} + z U_{yz} - yz U_{yz} - zU_z B_2 \\
			+ U_y (C_1 - (B_2 + I)y) -  yA_2 U_y - A_2 U B_2 = 0,\\
		z(1-z)U_{zz} + x U_{xz} + y U_{yz} - yz U_{yz} - yU_y B_3 \\
	+ U_z (C_1 - (B_3 + I)z) -  zA_2 U_z - A_2 U B_3 = 0;
		\end{array}$ &  $\begin{array}{c}
		A_1A_2 = A_2A_1,\\
	B_iB_j = B_jB_i,\\
	B_i C_1 = C_1 B_i 
		\end{array}$\\
		\hline
		$F_8$ & $\begin{array}{c}
		x(1-x)U_{xx} - xyU_{xy} - xz U_{xz} - x A_1U_x + \\
	U_x (C_1 - (B_1 + I)x)	- (yU_y + zU_z) B_1 - A_1UB_1 = 0,\\
		y(1-y)U_{yy} - xyU_{xy} - yz U_{yz} + zU_{yz}  - yA_1U_y\\
		 - (xU_x + zU_z) B_2 + U_y (C_2 - (B_2 + I)y) - A_1UB_2 = 0,\\	
	z(1-z)U_{zz} - xzU_{xz} - yz U_{yz} + yU_{yz}  - zA_1U_z\\
	 - (xU_x + yU_y) B_3 + U_z (C_2 - (B_3 + I)z) - A_1UB_3 = 0;
		\end{array}$ &  $\begin{array}{c}
		B_iB_j = B_jB_i,\\
		C_1C_2 = C_2C_1,\\
		B_i C_j = C_j B_i 
		\end{array}$\\
		\hline
		$F_{10}$ & $\begin{array}{c}
		x(1-x)U_{xx} - 2xzU_{xz} - z^2 U_{zz} - zU_z(B_1 + I)\\
	+ U_x (C_1 - (B_1 + I)x) - A_1(xU_x + zU_z)
	 - A_1 UB_1 = 0,\\
		y(1-y)U_{yy} + zU_{yz} + U_y (C_2 - (B_2 + I)y)\\
		-yA_2U_y - A_2 UB_2 = 0,\\
		z(1-z)U_{zz} - 2xzU_{xz} + y U_{yz} - x^2u_{xx} + U_z C_2 \\
 - (xU_x + zU_z)(B_1 + I) 
- A_1(xU_x + zU_z) - A_1 UB_1 = 0;	
	\end{array}$ &  $\begin{array}{c}
		A_1A_2 = A_2A_1,\\
		 B_1B_2 = B_2B_1,\\
		 C_1C_2 = C_2C_1,\\
		 B_i C_j = C_j B_i
		\end{array}$\\
		\hline
		$F_{11}$ & $\begin{array}{c}
		x(1-x)U_{xx} - xzU_{xz} + U_x (C_1 - (B_1 + I)x)\\
	 - A_1(xU_x + zU_z) - A_1 UB_1 = 0,\\
	 	y(1-y)U_{yy} - yzU_{yz} + z U_{yz}  - zU_zB_2\\
	 + U_y (C_2 - (B_2 + I)y) - yA_2U_y - A_2 UB_2 = 0,\\
	 	z(1-z)U_{zz} -  xyU_{xy} -  yzU_{yz}- xzU_{xz} + yU_{yz} - y U_y B_1\\
	 	 + U_z (C_2 - (B_1 + I)z) - A_2(xU_x + zU_z)
	 	  - A_2 UB_1 = 0;
		\end{array}$ &  $\begin{array}{c}
		A_1A_2 = A_2A_1,\\
	B_1B_2 = B_2B_1,\\
	C_1C_2 = C_2C_1,\\
	B_i C_j = C_j B_i
		\end{array}$\\
		\hline
		$F_{12}$ & $\begin{array}{c}
		x(1-x)U_{xx} -  xyU_{xy} -  yzU_{yz}- xzU_{xz} - z U_z B_1\\
	+ U_x (C_1 - (B_1 + I)x) - A_1(xU_x + yU_y)
	 - A_1 UB_1 = 0,\\
		y(1-y)U_{yy} -  xyU_{xy} +  zU_{yz} + U_y (C_2 - (B_1 + I)y)\\
		- A_2(xU_x + yU_y) - A_2 UB_1 = 0,\\
	z(1-z)U_{zz} -  xzU_{xz} +  yU_{yz} + U_z (C_2 - (B_2 + I)z)\\
		-xU_x B_2 - zA_1U_z - A_1 UB_2 = 0;
		\end{array}$ &  $\begin{array}{c}
		A_1A_2 = A_2A_1,\\
	B_1B_2 = B_2B_1,\\
	C_1C_2 = C_2C_1,\\
	B_i C_j = C_j B_i
		\end{array}$\\
		\hline
		$F_{13}$ & $\begin{array}{c}
		x(1-x)U_{xx} +  yU_{xy} + zU_{xz} - xzU_{xz}\\
	+ U_x (C_1 - (B_1 + I)x) - A_1(xU_x + zU_z)
	 - A_1 UB_1 = 0,\\
	y(1-y)U_{yy} +  xU_{xy} + zU_{yz} - yzU_{yz} - zU_z B_2\\
	+ U_y (C_1 - (B_2 + I)y) - yA_2U_y - A_2 UB_2 = 0,\\
	z(1-z)U_{zz} +  xU_{xz} +  yU_{yz} -  xyU_{xy} -  yzU_{yz}\\
	- xzU_{xz} + U_z (C_1 - (B_1 + I)z) -yU_y B_1\\ 
	- A_2(xU_x + zU_z) - A_2 UB_1 = 0;
		\end{array}$ &  $\begin{array}{c}
		A_1A_2 = A_2A_1,\\
	B_1B_2 = B_2B_1,\\
	B_i C_1 = C_1 B_i
		\end{array}$\\
		\hline
		$F_{14}$ & $\begin{array}{c}
		x(1-x)U_{xx} -  xyU_{xy} - yzU_{yz} - 2xzU_{xz} - z^2U_{zz}\\
	+ U_x (C_1 - (B_1 + I)x) - zU_z - (yU_y + zU_z)B_1 \\
	- A_1(xU_x + zU_z) - A_1 UB_1 = 0,\\
	y(1-y)U_{yy} -  xyU_{xy} + zU_{yz} - yzU_{yz} - zU_z B_2\\
	- xU_xB_2 + U_y (C_2 - (B_2 + I)y) - yA_1U_y
	 - A_1 UB_2 = 0,\\
	z(1-z)U_{zz} -2xzU_{xz} +  yU_{yz} -  xyU_{xy} -  yzU_{yz}\\
	+ U_z (C_2 - (B_1 + I)z) - x^2 U_{xx} - xU_x(B_1+I)\\
	 -yU_y B_1 - A_1(xU_x + zU_z)	- A_1 UB_1 = 0;
		\end{array}$ &  $\begin{array}{c}
	B_1B_2 = B_2B_1,\\
	C_1C_2 = C_2C_1,\\
	B_i C_j = C_j B_i
		\end{array}$\\
		\hline
		$H_{\mathcal{A}}$ & $\begin{array}{c}
	x(1-x)U_{xx} -  xyU_{xy} - yzU_{yz} - xzU_{xz} - zU_zB \\
	+ U_x (C - (B + I)x)  - A(yU_y + xU_x) - AUB = 0, \\
	y(1-y)U_{yy} -  xyU_{xy} -xz U_{xz} - yzU_{yz} + zU_{yz}  - xU_xB'\\
		 + U_y (C' - (B' + I)y) - (yU_y + zU_z)B
		  -  UBB' = 0,\\
		z(1-z)U_{zz} -xzU_{xz} -  xyU_{xy} -  yzU_{yz} + yU_{yz} - xU_{x} B'\\
		+ U_z (C' - (B' + I)z) - A(yU_y + zU_z)
			- A UB' = 0;
		\end{array}$ &  $\begin{array}{c}
		B, B', C, C'\\
		 \text{ are commuting }\\
		\end{array}$\\
		\hline
		$H_{\mathcal{B}}$ & $\begin{array}{c}
	x(1-x)U_{xx} -  xyU_{xy} - yzU_{yz} - xzU_{xz} - zU_zB \\
	+ U_x (C - (B + I)x)  - A(yU_y + xU_x) - AUB = 0, \\
	y(1-y)U_{yy} -  xyU_{xy} -xz U_{xz} - yzU_{yz} 	- xU_xB'\\
 + U_y (C' - (B' + I)y) - (yU_y + zU_z)B -  UBB' = 0,\\
	z(1-z)U_{zz} -xzU_{xz} -  xyU_{xy} -  yzU_{yz}  - xU_{x} B'\\
	+ U_z (C'' - (B' + I)z) - A(yU_y + zU_z) - A UB' = 0;
		\end{array}$ &  $\begin{array}{c}
		B, B', C, C', \text{ and }\\
		C'' \text{ are commuting }\\
		\end{array}$\\
		\hline
		$H_{\mathcal{C}}$ & $\begin{array}{c}
			x(1-x)U_{xx} -  xyU_{xy} - yzU_{yz} - xzU_{xz}\\
			 + yU_{xy}+ zU_{xz}	-zU_zB	+ U_x (C - (B + I)x) \\
			  - A(yU_y + xU_x)  - AUB = 0, \\
		y(1-y)U_{yy} -  xyU_{xy} -xz U_{xz} - yzU_{yz}\\
		 + xU_{xy}+ zU_{yz}	- xU_xB' + U_y (C - (B' + I)y) \\
		 - (yU_y + zU_z)B -  UBB' = 0,\\
		z(1-z)U_{zz} -xzU_{xz} -  xyU_{xy} -  yzU_{yz}\\
		 + xU_{xz}+ yU_{yz}	+ U_z (C - (B' + I)z) - xU_{x} B'\\
		  - A(yU_y + zU_z) - A UB' = 0.
		\end{array}$ &  $\begin{array}{c}
		B, B' \text{ and } C\\
		 \text{are commuting }\\
		\end{array}$\\
		\hline
		\caption{Systems of partial matrix differential equations of bilateral type satisfied by Lauricella and Srivastava matrix functions}
\end{longtable}}

\begin{example} {\rm We show here that the conditions on the matrices given in the third column of Table~1 are necessary for systems of matrix differential equations of bilateral type to hold. 
		
		Consider the differential equation of the bilateral type
		\begin{align}
		&x(1-x)U_{xx} - xzU_{xz} + U_x(C_1 - (B_1+I)x) - A_1(xU_x + zU_z) - A_1UB_1 = 0\label{e3.19}
		\end{align}
		satisfied by the Lauricella matrix function $F_3$. Let $U = F_{3}(A_1, A_2, A_2, B_1, B_2, B_1; C_1, C_2, C_3; x, y, z) = \sum_{m,n, p=0}^{\infty} U_{m,n, p} \, x^m y^n z^p$. Then, using the partial derivatives \eqref{e3.17}, we have
			\begin{align}
		&x(1-x)U_{xx} - xzU_{xz} + U_x C_1\nonumber\\
		& = \sum_{m,n, p=0}^{\infty} (m+1) U_{m+1,n, p} \, x^m y^n z^p (C_1 + mI) - \sum_{m,n, p=0}^{\infty} m(m-1) U_{m,n, p} \, x^m y^n z^p\nonumber\\
		& \quad - \sum_{m,n, p=0}^{\infty} m\,p\, U_{m,n, p} \, x^m y^n z^p
		\end{align}
		and 
		\begin{align}
		& x U_x (B_1+I) + A_1(xU_x + zU_z) + A_1UB_1\nonumber\\
		& = \sum_{m,n, p=0}^{\infty} m U_{m,n, p} \, x^m y^n z^p (B_1 + I) + \sum_{m,n, p=0}^{\infty} A_1 (m+p) U_{m,n, p} \, x^m y^n z^p\nonumber\\
		& \quad + \sum_{m,n, p=0}^{\infty} A_1 U_{m,n, p} \, x^m y^n z^p B_1.
		\end{align}
			In particular for $m=1$ and $n=p=0$, we get
\begin{align}
&x(1-x)U_{xx} - xzU_{xz} + U_x C_1 = A_1 (A_1 + I) B_1 (B_1 + I) C_1^{-1} x\label{e3.22}
\end{align}
and 
\begin{align}
& x U_x (B_1+I) + A_1(xU_x + zU_z) + A_1UB_1 = A_1 (A_1 + I) B_1 C_1^{-1} x (B_1 + I).\label{e3.23}
\end{align}	
From \eqref{e3.22} and \eqref{e3.23}, one can conclude that the equation \eqref{e3.19} will hold if
\begin{align}
A_1 (A_1 + I) B_1 (B_1 + I) C_1^{-1} x = A_1 (A_1 + I) B_1 C_1^{-1} x (B_1 + I).
\end{align}	
This gives $B_1 C_1 = C_1 B_1$. Similarly, at $m = 0, n = 1, p = 0$, we have 
\begin{align}
&x(1-x)U_{xx} - xzU_{xz} + U_x C_1 = A_1 A_2 B_1 B_2 C_1^{-1} C_2^{-1} C_1 y\label{e4.22}
\end{align}
and 
\begin{align}
& x U_x (B_1+I) + A_1(xU_x + zU_z) + A_1UB_1 = A_1 A_2 B_2 C_2^{-1}  yB_1,\label{e4.23}
\end{align}
which gives $B_1 B_2 = B_2 B_1$, $C_1 C_2 = C_2 C_1$ and $B_1 C_2 = C_2 B_1$. The remaining conditions on the matrices can be obtained by changing the particular values of $m$, $n$ and $p$. 
}
\end{example}
In Table~2, we give the integral representations of Lauricella matrix functions of three variables and Srivastava's triple hypergeometric matrix functions. 
\begin{longtable}{|l|l |c|}
	\hline
	Functions & Integral Representations & Conditions on Matrices\\
	\hline
	$F_6$ & $\begin{array}{c}\\
	\Gamma{\left(\begin{array}{c}
		C_1, C_2\\
		A_1, A_2, A_3, C_1 - A_1, C_2-A_2-A_3
		\end{array}\right)} \iiint u^{A_1-I}\\
	\times v^{A_2-I} w^{A_3-I} (1-u)^{C_1-A_1-I} (1-v-w)^{C_2-A_2-A_3-I}\\
	\times (1-vy)^{-B_2} (1-ux-wz)^{-B_1} \ du \, dv\, dw, \,	0\leq u \leq 1,\\
	\,  v \geq 0, \, w\geq 0, \ v+w \leq 1,  r + t < 1, \ \vert x\vert
	\leq r, \vert z\vert \leq t,\\
	\beta(A_1) >0, \beta(A_2)>0, \beta(A_3)>0, \beta(C_1)>0, \\
	\beta(C_2)>0, \beta(C_1-A_1)>0, \beta(C_2-A_2-A_3)>0.
	\\
	\end{array}$ & $\begin{array}{c}
	A_iA_j = A_jA_i,\\
B_iC_j = C_jB_i,\\
C_1C_2 = C_2C_1,\\
A_i C_j = C_j A_i
	\end{array} $  \\
	\hline
	$F_7$ & $\begin{array}{c}\\
	\iiint (1-ux)^{-A_1} (1-vy-wz)^{-A_2} u^{B_1-I} v^{B_2-I} w^{B_3-I}\\ 
	\times (1-u-v-w)^{C_1-B_1-B_2-B_3-I}  du \, dv \, dw\\
	\times 	\Gamma{\left(\begin{array}{c}
		C_1\\
		B_1, B_2, B_3, C_1 - B_1-B_2-B_3
		\end{array}\right)}, u\geq 0,\\
	\, v\geq 0, \, w\geq 0, \, u + v+ w \leq 1, \, s+t < 1, \, \vert y\vert \leq s,\\
	\vert z\vert \leq t,	\beta(B_1) >0, \beta(B_2)>0, \beta(B_3)>0, \beta(C_1)>0, \\
	\beta(C_1-B_1-B_2-B_3)>0;\\
	\end{array}$ & $\begin{array}{c}
	B_iB_j = B_jB_i,\\
	B_iC_1 = C_1B_i
	\end{array}$\\
	\hline
	$F_8$ & $\begin{array}{c}\\
	\iiint (1-ux-vy-wz)^{-A_1} u^{B_1-I} v^{B_2-I} w^{B_3 - I}\\
	\times (1-u)^{C_1 - B_1 - I} (1-v-w)^{C_2-B_2-B_3-I}  du\,dv \, dw,\\
	\times	\Gamma{\left(\begin{array}{c}
		C_1, C_2\\
		B_1, B_2, B_3, C_1 - B_1, C_2-B_2-B_3
		\end{array}\right)}, 0 \leq u \leq 1,\\
	v \geq 0, \ w \geq 0, \ v + w \leq 1, \ r + s + t < 1, \, \vert x\vert \leq r, \vert y\vert \leq s, \\ \vert z\vert \leq t, \beta(B_1) >0, \beta(B_2)>0, \beta(B_3)>0, \beta(C_1)>0,\\
	\beta(C_2)>0, \beta(C_1-B_1)>0, \beta(C_2-B_2-B_3)>0.
	\end{array}$ & $\begin{array}{c}
	B_iB_j = B_jB_i,\\
		C_1C_2 = C_2C_1,\\
	B_iC_j = C_jB_i
	\end{array}$\\
	\hline
	$F_{11}$ & $\begin{array}{c}\\
	\Gamma{\left(\begin{array}{c}
		C_1, C_2\\
		A_1, A_2, C_1 - A_1, C_2-A_2
		\end{array}\right)} \int_{0}^{1} \int_{0}^{1} u^{A_1-I} v^{A_2-I}\\
	 \times (1-u)^{C_1-A_1-I}(1-v)^{C_2-A_2-I}(1-ux-vz)^{-B_1}  \\
	\times (1-vy)^{-B_2} \,du \, dv, \ r+t < 1, \ \vert x\vert \leq r, \vert z\vert \leq t,\\
	\beta(A_1)>0, \beta(A_2)>0, \beta(C_1)>0, \beta(C_2)>0,\\
	\beta(C_1-A_1)>0, \beta(C_2-A_2)>0.
	\\ \end{array}$ & $\begin{array}{c}
	A_1A_2 = A_2A_1,\\
	B_iC_j = C_jB_i,\\
	C_1C_2 = C_2C_1,\\
	A_i C_j = C_j A_i
	\end{array} $ \\
	\hline
	$F_{12}$ & $\begin{array}{c}\\
	\iiint (1-vy)^{A_1} (1-ux-vy-wz+vwyz)^{-A_1} u^{B_1-I} v^{A_2-I} \\
	\times w^{B_2-I} (1-u)^{C_1-B_1-I}   (1-v-w)^{C_2-A_2-B_2-I} (1-vy)^{-B_1} \\
	\times  \, du\,  dv\, dw 	\Gamma{\left(\begin{array}{c}
		C_1, C_2\\
		A_2, B_1, B_2,  C_1 - B_1, C_2-A_2-B_2
		\end{array}\right)},\\
	0 \leq u \leq 1, v \geq 0, w \geq 0, v+w \leq 1, 	r+s+t < 1+st, \\ 
	\vert x\vert \leq r, \vert y\vert \leq s, \vert z\vert \leq t, \beta(A_2) >0,  \beta(B_1)>0,\\
	\beta(B_2)>0, \beta(C_1)>0, \beta(C_2) >0,  \beta(C_1-B_1)>0, \\
	\beta(C_2- A_2-B_2)>0.
	\\ \end{array}$ &  $\begin{array}{c}
	C_1, C_2, B_1, B_2 \text{ and } \\
A_2 \,	\text { are commuting }
	\end{array}$\\
	\hline
	$F_{13}$ & $\begin{array}{c}\\
	\iint (1-ux)^{-A_1} (1-vy - uz)^{-A_2}  u^{B_1-I} v^{B_2-I}\\
	\times (1-u-v)^{C_1-B_1-B_2-I} \, du \, dv\\
	\Gamma{\left(\begin{array}{c}
		C_1\\
		B_1, B_2, C_1 - B_1-B_2
		\end{array}\right)},  u\geq 0, \, v\geq 0, \\
	u + v \leq 1, \ s + t < 1, \ \vert y\vert \leq s, \vert z\vert \leq t, 	\beta(B_1) >0,\\
	\beta(B_2)>0, \beta(C_1)>0, \beta(C_1-B_1-B_2)>0.
	\\ \end{array}$ & $\begin{array}{c}
	C_1, B_1 \text{ and } B_2\\
	\text { are commuting }
	\end{array} $\\
	\hline
	$H_{\mathcal{A}}$ & $\begin{array}{c}\\
	\int_{0}^{1} \int_{0}^{1} (1-ux-vy-vz+v^2yz)^{-A} (1-vy)^A u^{B-I}  \\
	\times v^{B'-I} (1-u)^{C-B-I} (1-v)^{C'-B'-I} (1-vy)^{-B}\\
	du \, dv \Gamma{\left(\begin{array}{c}
		C, C'\\
		B, B', C - B, C'-B'
		\end{array}\right)}, 	r+s+t < 1+st,\\
	\ \vert x\vert \leq r,  \vert y\vert \leq s, \vert z\vert \leq t, 	\beta(B) >0, \beta(B')>0, \\
	\beta(C)>0, \beta(C')>0, \beta(C-B)>0, \beta(C'-B')>0.
	\\ \end{array}$ & $\begin{array}{c}
	B, B', C \text{ and } C'\\
	\text { are commuting }
	\end{array} $\\
	\hline
	$H_{\mathcal{B}}$ & $\begin{array}{c}\\
	\Gamma^{-1}(A) \Gamma^{-1} (B) \Gamma^{-1}(B') \int_{0}^{\infty}\int_{0}^{\infty} \int_{0}^{\infty} e^{-(u+v+w)} v^{A-I}\\ 
	\times	u^{B-I} w^{B'-I} {}_0F_1(-;C;xuv) {} \ _0F_1(-;C';yuw) {} \\
	_0F_1(-;C'';zvw) \ du \ dv \ dw, 	\vert x\vert \leq r,  \vert y\vert \leq s,  \vert z\vert \leq t,\\
	\max\{r, s, t\}<1, \beta(A)>0, \beta(B)>0, \beta(B')>0.\\
	\end{array}$ & $\begin{array}{c}
	B, B' \text{ commutes }\\
	\text { with } A
	\end{array} $\\
	\hline
	$H_{\mathcal{C}}$ & $\begin{array}{c}\\
	\Gamma{\left(\begin{array}{c}
		C\\
		A,	B, C-A-B
		\end{array}\right)}\int_{0}^{1} \int_{0}^{1} u^{A-I} v^{B-I} \\
	(1-u)^{C-A-I} (1-v)^{C-A-B-I} (1-ux)^{-B} (1-ux)^{B'}\\
	(1-ux-vy-wz+uvy-zxu^2)^{-B'} du \, dv,\\
	\ r+s+t+rt < 1+s, \ \vert x\vert \leq r, \vert y\vert \leq s, \vert z\vert \leq t,\\
	\beta(A)>0, \beta(B)>0, \beta(C)>0, \beta(C-A-B)>0. 
	\\ \end{array}$ & $\begin{array}{c}
	A,	B, C \text{  commutes }\\
	\text { and } B'C = CB'
	\end{array} $\\
	\hline
	\caption{Integral Representations of Lauricella and Srivastava Matrix Functions}
\end{longtable}

%


\begin{thebibliography}{99}
\bibitem{ma} M.\,Abdalla, \textit{Special matrix functions: characteristics, achievements and future directions},  Linear Multilinear Algebra (2020), no. 1, 1--28. 

\bibitem{al} A.\,Altin, B.\,Cekim, R.\,Sahin, \textit{On the matrix versions of Appell hypergeometric functions}. Quaest. Math. 37 (2014), no. 1, 31--38.


\bibitem{hb} H.\,Buchholz, \textit{The confluent hypergeometric function}, Springer-Verlag, New York, 1969.

	
\bibitem{ds57} N.\,Dunford, J.\,Schwartz, \textit{Linear Operators}, Part-I, New York: Addison-Wesley, 1957.	
		
\bibitem{ds1} R.\,Dwivedi, V.\,Sahai, \textit{On the hypergeometric matrix functions of two variables.} Linear Multilinear Algebra 66 (2018), no. 9, 1819--1837.

\bibitem{ds4} R.\,Dwivedi, V.\,Sahai, \textit{On the hypergeometric matrix functions of several variables,} J. Math. Phys. 59 (2018), no. 2, 023505, 15pp. 

\bibitem{ds5} R.\,Dwivedi, V.\,Sahai, {A note on the Appell matrix functions}. Quaest. Math. 43 (2020), no. 3, 321--334.



\bibitem{ex} H.\, Exton, \textit{On certain confluent hypergeometric functions of three variables}, Ganita 21 (1970), no. 2, 79--92.

			
\bibitem{gl} G.H.\,Golub, C.F.\,van Loan, \textit{Matrix Computations}, Johns Hopkins Univ. Press, Baltimore, MD, 1989.
	
\bibitem{hl04} G.D.\,Hu, M.\,Liu, \textit{The weighted logarithmic matrix norm and bounds of the matrix exponential.}, Linear Algebra Appl. 2004;390:145--154.
%
	
\bibitem{jjc98a} L.\,Jodar, J.C.\,Cortes, \textit{Some properties of gamma and beta matrix functions.} Appl. Math. Lett. 11 (1998), no. 1, 89--93.
	
\bibitem{jjc98b}  L.\,Jodar, J.C.\,Cortes, \textit{On the hypergeometric matrix function.} Proceedings of the VIIIth Symposium on Orthogonal Polynomials and Their Applications (Seville, 1997). J. Comput. Appl. Math. 99 (1998), no. 1-2, 205--217.

\bibitem{jc00}  L.\,Jodar, J.C.\,Cortes, \textit{Closed form general solution of the hypergeometric matrix differential equation}. Mathematical and computer modelling in engineering sciences. Math. Comput. Modelling 32 (2000), no. 9, 1017--1028. 


\bibitem{m68} W. Miller, \textit{Lie Theory and Special Functions}, Academic Press, New York, 1968.

%
%
%

%
\bibitem{hm84}  H.\,M.\,Srivastava, H.\,L.\,Manocha, \textit{A Treatise on Generating Functions}, Ellis Horwood Series: Mathematics and its Applications. Ellis Horwood Ltd., Chichester; Halsted Press [John Wiley \& Sons, Inc.], New York, 1984.	
	
\bibitem{te05} L.N.\,Trefethen, M.\,Embree, \textit{Spectra and pseudospectra: the behaviour of nonnormal matrices and operators.} Princeton, NJ: Princeton University press; 2005. 	
	
\bibitem{vl} C.\,Van Loan, \textit{The sensitivity of the matrix exponential}, SIAM J. Numer. Anal. 14 (1977), no. 6, 971--981.	
\end{thebibliography}
\end{document}